\documentclass[12pt,oneside]{article}
\usepackage{amsmath,amssymb,amsfonts,amsthm}
\usepackage{color}
\usepackage[all]{xy}
\usepackage{graphicx}
\usepackage{color}
\usepackage{makeidx}
\usepackage{multirow}
\usepackage{maplestd2e}
\usepackage{rotating}

\newcommand{\T}{{\cal T}}

\newcommand{\Real}{\mathbb R}

\newcommand{\tm}{\T M}

\numberwithin{equation}{section} 
\numberwithin{figure}{section} 

\theoremstyle{plain}
\newtheorem*{theorem*}{Theorem}
\newtheorem{theorem}{Theorem}[section]
\newtheorem{lemma}[theorem]{Lemma}

\newtheorem{corollary}[theorem]{Corollary}
\newtheorem{remark}[theorem]{Remark}

\newtheorem{definition}[theorem]{Definition}
\newtheorem{property}[theorem]{Property}

\newtheorem{example}{Example}
\newtheorem*{acknowledgement*}{Acknowledgement}

\numberwithin{equation}{section}

\newcommand\overcirc[1]{\raisebox{10pt}{\tiny{$\circ$}}{\kern-7.5pt}\mbox{$#1$}}
\newcommand\undersym[2]{\raisebox{-6pt}{$#2$}{\kern-5pt}\mbox{$#1$}}
\newcommand\overdiamond[1]{\raisebox{10pt}{\small$\star$}{\kern-7.5pt}\mbox{$#1$}}
\newcommand\overast[1]{\raisebox{10pt}{\small$\ast$}{\kern-7.5pt}\mbox{$#1$}}
\newcommand\overlind[1]{\raisebox{10pt}{\small$\overline{{\hspace{2pt}}\star}$}{\kern-7.5pt}\mbox{$#1$}}
\newcommand\overlinc[1]{\raisebox{10pt}{\small$\overline{{\hspace{2pt}}\circ}$}{\kern-7.5pt}\mbox{$#1$}}
\newcommand\overlina[1]{\raisebox{10pt}{\small$\overline{{\hspace{1pt}}\ast}$}{\kern-7.5pt}\mbox{$#1$}}

\begin{document}

\title{{\bf{ Semi Concurrent vector fields in Finsler geometry}}}
\author{{\bf Nabil L. Youssef$\,^1$, S. G. Elgendi$^2$ and Ebtsam H. Taha$^1$}}
\date{}

\maketitle
\vspace{-1.15cm}

\begin{center}
{$^1$Department of Mathematics, Faculty of Science,\\
Cairo University, Giza, Egypt}
\vspace{-8pt}
\end{center}

\begin{center}
{$^2$Department of Mathematics, Faculty of Science,\\
Benha University, Benha, Egypt}
\vspace{-8pt}
\end{center}

\begin{center}
nlyoussef@sci.cu.edu.eg, nlyoussef2003@yahoo.fr\\
salah.ali@fsci.bu.edu.eg, salahelgendi@yahoo.com\\
ebtsam.taha@sci.cu.edu.eg, ebtsam.h.taha@hotmail.com

\end{center}

\vspace{0.5cm}
\noindent{\bf Abstract.} In the present paper, we introduce and investigate the notion of a semi concurrent vector field on a Finsler manifold. We show that some special Finsler manifolds admitting such vector fields turn out to be Riemannian. We prove that Tachibana's characterization of  Finsler manifolds admitting a concurrent vector field leads to Riemannain metrics.
We give an answer to the question raised in \cite{DWF}: "Is any n-dimensional Finsler manifold $(M,F)$, admitting a non-constant smooth function $f$ on $M$ such that $\frac{\partial f}{\partial x^i}\frac{\partial g^{ij}}{\partial y^k}=0$, a Riemannian manifold?\,". Various examples for conic Finsler and Riemannian spaces that admit semi-concurrent vector field are presented.
Finally, we conjectured that there is no regular Finsler non-Riemannian metric that admits a semi-concurrent vector field. In other words, a Finsler metric admitting a semi-concurrent vector field is necessarily either Riemannian or conic Finslerian.




\section{\textbf{Introduction}}

Concurrent vector fields in Riemannian geometry and related topics have been studied before by many authors, see for example \cite{Petrovic,Sasaki,Yano}.

In 1950, Tachibana \cite{Tachibana} generalized  the notion of a concurrent vector field from Riemannian geometry to Finsler geometry and characterized  the spaces admitting this kind of vector fields.

In 1974, Mastumoto and Eguchi \cite{meguchi} discussed the geometric consequences of the existence of  concurrent vector fields on  Finsler manifolds. They showed that a concurrent vector field controls the geometry of the underlying manifold.

Two years later, in 1976, Hashiguchi \cite{hashiguchi} treated a special kind of conformal change in Finsler geometry introducing the concept of C-conformal change. Hashiguchi was able to show that some of the results obtained by Matsumoto and Eguchi in the case when a Finsler manifold admits a concurrent vector field also hold when this space admit a C-conformal change.

In 2012, Youssef  et. al. \cite{beta2}, have introduced the concept of a B-condition which generalizes all the above mentioned  concepts. They have shown that some of the results previously obtained remain valid in this more general setting.

In 2012, Peyghan and Tayebi \cite{DW_Tayebi}  have proved that if ${M_1}_{f_2}\times _{f_1} M_2$ is a   doubly warped Finsler manifold, with  $f_1$ constant on $M_1$ and $f_2$ constant on $M_2$, then ${M_1}_{f_2}\times _{f_1} M_2$ is a Berwaldian manifold if and only if $M_1$ is Riemannain, $M_2$ is Berwaldian and \,$\dfrac{\partial f_1}{\partial x^i}\,C^{ij}_k=0.$

Recently, in 2015, Faghfouri and Hosseninoghi \cite{DWF} have treated the question: Does there exist a non-constant smooth function $f$ on a Finsler manifold $M$ such that $\dfrac{\partial f}{\partial x^i}\,C^{ij}_k=0$?  They showed that any two dimensional Finsler manifold admitting such a kind of function is necessarily Riemannian. They conjectured that this should hold for a Finsler manifold of arbitrary dimension. This problem, in a more general setting, is one of the main objects of the present paper.

In this paper, we investigate the relation between all the above mentioned concepts. We focus our attention to the most general concept, which we call \textit{semi-concurrent vector field}. According to Tashibana's theorem, we prove that a regular Finsler manifold which admits a concurrent vector field is Riemannian. We study some special Finsler manifolds admitting a semi-concurrent vector field. Various examples of non-Riemannain conic Finsler spaces admitting semi-concurrent vector fields are given. We investigate the cases under which an n-dimensional Finsler manifold $(M,F)$ admitting a non-constant smooth function $f$ on $M$ such that $\frac{\partial f}{\partial x^i}\frac{\partial g^{ij}}{\partial y^k}=0$ is a Riemannian manifold, giving an answer to the question of \cite{DWF}. Finally, we conjuncture that there is no regular Finsler metric admitting a semi-concurrent vector field. In other words, a Finsler metric admitting a semi-concurrent vector field is necessarily either Riemannian or conic Finslerian.


\section{\textbf{Notations and preliminaries}}
Let  $(M,F)$ be  an n-dimensional smooth connected Finsler
manifold; $F$ being the Finsler function (Finsler metric or Lagrangian). Let $(x^i)$ be
the coordinates of any  point of the base manifold $M$  and $(y^i)$ a
supporting element at the same point.  We use the following terminology and notations:\\
   $\partial_i$: partial differentiation with respect to $x^i$,\\
   $\dot{\partial}_i$:  partial differentiation
    with respect to  $y^i$ (basis vector fields of the vertical bundle),\\
   $g_{ij}:=\frac{1}{2}\dot{\partial}_i\dot{\partial}_j F^2=\dot{\partial}_i
   \dot{\partial}_jE$:  the
   Finsler metric tensor, where $E:=\frac{1}{2}F^2$ is the energy function,\\
   $l_i:=\dot{\partial}_iF=g_{ij} l^j=g_{ij}\frac{y^j}{F}$: the
    normalized supporting element; $l^i:=\frac{y^i}{F}$,\\
 $l_{ij}:=\dot{\partial}_il_j$,\\
   $h_{ij}:=Fl_{ij}=g_{ij}-l_il_j$:  the angular metric tensor,\\
   \vspace{7pt}$C_{ijk}:=\frac{1}{2}\dot{\partial}_kg_{ij}=\frac{1}{4}\dot{\partial}_i
    \dot{\partial}_j\dot{\partial}_k F^2$:  the Cartan  tensor, \\
      $C^i_{jk}:=g^{ri}C_{rjk}=\frac{1}{2}g^{ri}\dot{\partial}_kg_{rj}$:
     the h(hv)-torsion tensor,\\
   $G^i$: the components of the geodesic spray associated
    with $(M,F)$,\\
    $ N^i_j:=\dot{\partial}_jG^i$: the Barthel (or Cartan nonlinear) connection
    associated with $(M,F)$,\\
    $\delta_i:=\partial_i-N^r_i\dot{\partial}_r$: the basis
     vector fields of the horizontal bundle,\\
     \vspace{7pt}$G^i_{jh}:=\dot{\partial}_hN^i_j=\dot{\partial}_h\dot{\partial}_jG^i$:
     the coefficients of Berwald connection, \\
        \vspace{.3cm}$\Gamma^i_{jk}:=\frac{1}{2}g^{ir}(\delta_jg_{kr}+\delta_kg_{jr}-\delta_rg_{jk})$:
          the Christoffel symbols with respect to $\delta_i$,\\
       $(\Gamma^i_{jk},N^i_j,C^i_{jk})$: The Cartan connection
        $C\Gamma$.

\vspace{5pt}
    For the Cartan  connection $(\Gamma^i_{jk},
         N^i_j,C^i_{jk})$, we define\\
         $X^i_{j\mid k}:=\delta_kX^i_j+X^m_j\Gamma^i_{mk}-X^i_m\Gamma^m_{jk}$: the
    horizontal covariant derivative of $X^i_j$,\\
   $ X^i_j|_k:=\dot{\partial}_kX^i_j+X^m_jC^i_{mk}-X^i_mC^m_{jk}$: the
    vertical  covariant derivative of $X^i_j$.

\vspace{7pt}
Now, we give the definition we shall adopt for a Finsler manifold.

\begin{definition}\label{fin.struc.} A Finsler structure on a manifold $M$  is a   function
$$F:TM\rightarrow \mathbb{R}$$
with the following properties:
\begin{description}
   \item[(a)] $F\geqslant 0$ and $F(0)=0$.

    \item[(b)] $F$ is $C^\infty$ on the slit tangent
    bundle  $\tm:=TM\backslash\{0\}$.

    \item[(c)] $F(x,y)$ is positively homogenous of degree one in $y$: $F(x,\lambda y) = \lambda F(x,y)$
     for all $y \in TM$ and $\lambda > 0$.

    \item[(d)] The Hessian matrix ${g_{ij}(x,y):=\dot{\partial}_i\dot{\partial}_j E}$
is positive-definite at each point of $\tm$, where $E:=\frac{1}{2}F^2$ is the energy function of the Lagrangian $F$.
\end{description}
The pair $(M,F)$ is called  a {Finsler manifold} and the symmetric bilinear form $g=g_{ij}(x,y)dx^i\otimes dx^j$ is called the
Finsler metric  tensor  of the Finsler manifold $(M,F)$. \\
Sometimes, a function $F$ satisfying the above conditions is said to be a regular Finsler metric.
\\[0.2cm]
$\bullet$ When the metric tensor $g$ is non-degenerate at each point of $\tm$, the pair $(M,F)$ is called a pseudo-Finsler manifold.
\\[0.2cm]
$\bullet$ When $F$ satisfies the conditions (a)-(d) but only on an open conic subset $A$ of $TM$ (for every $v \in A$ and $\mu > 0, \mu v\in  A $), the pair $(A,F)$ is called  conic Finsler manifold.\\
If, moreover, the metric tensor $g$ is non-degenerate at each point of $A$, the pair $(A,F)$ is called  conic pseudo-Finsler manifold.
\end{definition}

For more details about conic Finsler and conic pseudo-Finsler metrics we refer, for example, to \cite{conic}.

\vspace{5pt}
In the following, we give the definitions of the special Finsler manifolds we shall use in the sequel.

\begin{definition}
A Finsler manifold $(M,F)$ is said to be Berwaldian if the Berwald tensor $G^{h}_{ijk}:=\dot{\partial}_i G^{h}_{jk}=\dot{\partial}_i\dot{\partial}_j\dot{\partial}_kG^h$ vanishes.
\end{definition}
It is to be noted that \cite{r2.22},
\begin{equation*}\label{Landsberg}
    G^{h}_{ijk}=0 \Longleftrightarrow   C^{h}_{ij{|} k } =0 \Longleftrightarrow G^{h}(x,\,y) \text{ is quadratic in}\, y \in T_{x}M.
\end{equation*}
\begin{definition}
A Finsler manifold $(M,F)$  is called Landsbergian if the Landsberg tensor $L_{ijk}:=\frac{1}{2}F\ell_hG^h_{ijk}$ vanishes.
\end{definition}
It is to be noted that \cite{r2.22},
\begin{equation*}\label{Landsberg}
    L_{ijk}=0 \Longleftrightarrow  y^{i} \, C^{h}_{jk{|} i } =0.
\end{equation*}

 \begin{definition} \em{\cite{C_2-like}}
A Finsler manifold $(M,L)$ of dimension $n\geq 2$ is said to be
$C_2$-like  if the Cartan  tensor $C_{ijk}$ satisfies
\begin{equation*}\label{c2-likedef.}
C_{ijk}=\frac{1}{C^2}C_iC_jC_k,
\end{equation*}
where $C_i:=C_{ijk}g^{jk}$ and $C^2:=C_iC^i$.
\end{definition}
\begin{definition} \em{\cite{r2.9}}
A Finsler manifold $(M,F)$ of dimension $n\geq 3$ is called C-reducible
if the  Cartan tensor $C_{ijk}$ has the form:
\begin{equation}\label{c-reducible*}
C_{ijk}=\frac{1}{n+1}(h_{ij}\,C_k+h_{ki}\,C_j+h_{jk}\,C_i).
\end{equation}
\end{definition}

\begin{definition} \em{\cite{r2.21}}\label{semi-1}
A Finsler manifold $(M,L)$ of dimension $n\geq 3$ is called
semi-C-reducible  if the Cartan tensor $C_{ijk}$ is written in the
form:
\begin{equation}\label{semi-c-red.}
C_{ijk}=\frac{r}{n+1}(h_{ij}C_k+h_{ki}C_j+h_{jk}C_i)+\frac{t}{C^2}C_iC_jC_k,
\end{equation}
 where $r$ and $t$ are scalar functions such that $r+t=1$.
\end{definition}


\section{{Semi-concurrent vector fields}}
Let $(M, F)$ be an $n$-dimensional smooth Finsler manifold.
\begin{definition}\em{\cite{Tachibana}} A vector field $X^i(x)$ on $M$ is said to be concurrent with respect to Cartan connection  if it satisfies
\begin{equation}\label{CVF}
X^{h}(x)\,C_{hij}=0, \quad {X^i}_{|_j}=-\delta^i_j.
\end{equation}
The condition \eqref{CVF} will be called C-condition.
\end{definition}

\begin{definition}\em{\cite{hashiguchi}}
The manifold $M$ fulfils the C-conformal condition if there exists on $M$ a conformal transformation $\overline{F}=e^{\sigma(x)} F$ such that
\begin{equation}\label{CC}
 \sigma_{h}(x)\,C^h_{ij}=0,
\end{equation}
where $\sigma_h:=\frac{\partial \sigma}{\partial x^h}$.
The condition \eqref{CC} will be called CC-condition.
\end{definition}
\begin{definition}\em{\cite{DWF}}
Assume that there exists  a non-constant smooth function $f$ on $M$ such that
\begin{equation}\label{FC}
    f_i(x)\frac{\partial g^{ij}}{\partial y^k}=0,
\end{equation}
where $f_i:=\frac{\partial f}{\partial x^i}$.
The condition \eqref{FC} will be called F-condition.
\end{definition}
\begin{definition}
 A vector field $B^i(x)$ on M is said to be semi-concurrent if it satisfies
\begin{equation}\label{SC}
 B^h(x)\,C_{hij}=0.
\end{equation}
The condition \eqref{SC} will be called the SC-condition.
\end{definition}

\begin{lemma}\label{FC-to-SC}
If a Finsler manifold satisfies the F-condition \eqref{FC}, then it satisfies the SC-condition \eqref{SC}.
\end{lemma}
\begin{proof}
Assume that $(M,F)$ satisfies \eqref{FC} so that
$$f_i(x) g^{ij}=b^j(x),$$
where $b^j$ are smooth functions on $M$. From which,
$$f_i=b^jg_{ij}.$$
Differentiating the above relation with respect to $y^k$, we get
$$ b^i(x)\,C_{ijk}=0.$$
This means that $(M,F)$ satisfies the SC-condition \eqref{SC}.
\end{proof}

\begin{remark}
\em{
The converse of the above result is not true in general. In fact, if $(M, F)$ satisfies the SC-condition, then, by \eqref{SC}, $\frac{\partial}{\partial y^k} (B^i\,g_{ij})=0$. Then, by integration, $B^i\,g_{ij}=\lambda_j(x)$ and  $B^i=\lambda_j(x)\,g^{ij}$. Now, by differentiation both sides with respect to $y^k$, we find that $\lambda_j(x)\,C^{ij}_k=0$. Therefore, the F-condition \eqref{FC} is satisfied only in the case when $\lambda_j(x)$ represents the gradient of a non-constant function $f\in C^\infty(M)$. This shows in particular that the SC-condition  is weaker than the F-condition.
}
\end{remark}
In view of Lemma \ref{FC-to-SC}, one can observe that the above mentioned conditions \eqref{CVF}-\eqref{SC} are interrelated as follows:
$$CC\text{-condition}\Longrightarrow F\text{-condition} \Longrightarrow SC\text{-condition},$$
$$ C\text{-condition} \Longrightarrow SC\text{-condition}.$$

Consequently, the $SC$-condition is the weakest condition and hence the most general one.
In the following we shall be concerned mainly with the $SC$-case: $B^{i}\, C_{ijk}=0.$ In fact, if a problem is solved in the $SC$-case, it would be also solved in the $CC$-, $F$- and $C$-cases. Moreover, the non-existence of a semi-concurrent vector field (the $SC$-condition is not satisfied) implies the non-existence of concurrent vector fields and the non-satisfaction of both the $CC$-condition and  the $F$-condition.
\begin{property}
In the F-case, the functions $f^i(x)$ defined by $f^i(x):=f_{k}g^{ik}$ are functions of $x$ only. Indeed,
$$\frac{\partial f^i}{\partial y^j}=\frac{\partial }{\partial y^j}(f_k(x)g^{ik})=f_k(x)\frac{\partial g^{ik}}{\partial y^j}=0.$$
\end{property}
 So when we lower (or raise) the index of $f^i$ (or $f_i$) the result is always functions of $x$ only. one can easily show that the same property is valid for the other three conditions .

\begin{lemma}\label{eguchi} For the  nonzero functions $B^i$ satisfying \eqref{SC}, if the  scalars $\alpha$ and
 $\alpha'$ satisfy
 \begin{equation}\label{independent lemma}  \alpha\, B^i+\alpha' y^i=0, \end{equation}
 then $\alpha=\alpha'=0$, which means that the two vector fields $B^{i}(x)$ and $y^{i}$ are independent.
 \end{lemma}

\begin{proof}
Contraction of $(\ref{independent lemma})$ by  $y_i$ and $B_i$, respectively, gives rise to the
 system:
$$B_0 \alpha+F^2\alpha'=0,$$
$$B^2\alpha+B_0 \alpha'=0,$$
where $B_0:=B_i\,y^i=B^i\,y_i$ and $B^2=B_i\,B^i$. We regard this system as a system of algebraic equations in the unknowns $\alpha$ and $\alpha'$.
Let us show first that $B_0$ and $B^2\,F^2 - B^2_0$ are nonzero.

Seeking for a contradiction, assume that $B_0=0$, so $B_i\,y^i=0$. Differentiation with respect to $y^j$ gives $B_j=0$. Since $M$ is connected, it follows that $B(x)$ is a constant function on $M$, which is a contradiction.

Now, assume that $B^2\,F^2\,-B^2_0=0$, then we have $B^2-\frac{B^2_0}{F^2}=0$. Differentiation both sides with respect to $y^i$, we get
\begin{equation*}
0=-\frac{1}{F^2}\left(2B_{0} \,\frac{\partial B_{0}}{\partial y^i}\right) + \frac{2}{F^3}\,B_0^{2}\,\frac{\partial F}{\partial y^i}  = -\frac{2}{F^2} B_{0} \,B_{i} + \frac{2}{F^3}\,B_0^{2}\,l_{i},\,\,  \text{  or  } \,\, B_i=\frac{B_0}{F}\,l_i.
\end{equation*}
Again, differentiating $B_i=\frac{B_0}{F}\,l_i$ with respect to $y^j$, we obtain
$$\frac{B_{j}}{F}\,l_{i} -\frac{B_{0}}{F^2}\,l_{i}\,l_{j} +\frac{B_{0}}{F^2}\,h_{ij}  =0,$$
and using $B_i=\frac{B_0}{F}l_i$, we get $\frac{B_0}{f^2}\,h_{ij}=0$, which is a contradiction. Hence, we have
\begin{equation}\label{lemmab}
  B_{0} \neq 0, \,\, B^2\,F^2 - B^2_0 \neq 0.
\end{equation}

Finally, since  $B^2F^2-B^2_0\neq 0$, the  above system has only the trivial solution; that is,  $\alpha=\alpha'=0$.
\end{proof}

\begin{theorem}\label{matsumoto}
 Let $(M,F)$ be a Finsler manifold. In each of the following cases
\begin{description}
    \item[(a)] $(M,F)$ is two-dimensional,

    \item[(b)] $(M,F)$ is three-dimensional such that $F(x,-y)=F(x,y)$,

    \item[(c)] $(M,F)$ is C-reducible,

    \item[(d)] $(M,F)$ is Berwaldian with $\det{({B^i}_{| j})}\neq 0$,
    \end{description}
if the Finsler manifold $(M,F)$ satisfies the SC-condition $(\ref{SC})$, then
it is Riemannian.
\end{theorem}

\begin{proof}~\par
\noindent\textbf{(a)} The Cartan tensor $C_{ijk}$ of a two-dimensional Finsler manifold is given by
$$FC_{ijk}=J\eta_i\eta_j\eta_k,$$
where $\eta_i$ is an orthogonal vector to $y^i$ and $J$ is the Berwald main scalar \cite{Berwald41}. Contracting by $B^i$, we have
$$J \,B^i\eta_i\eta_j\eta_k=0,$$
If $B^i\,\eta_i=0$, then this leads to $B^i=\mu y^i$, which contradicts Lemma
\ref{eguchi}. Hence, $J=0$  and so $C_{ijk}=0$.

\vspace{6pt}
\noindent\textbf{(b)} Making use of  Lemma \ref{eguchi}, the proof can be carried out in a similar manner as in \cite{meguchi} for concurrent vector fields.

\vspace{6pt}
\noindent\textbf{(c)} As $(M, F)$ is C-reducible, then by \eqref{c-reducible*} we have
$$C_{ijk}=h_{ij}C_k+h_{ki}C_j+h_{jk}C_i.$$
Contracting the above equation by $B^j B^k$,  we get $B^i\,B^j\,h_{ij}C_k=0.$ We have the following implication
\begin{eqnarray*}
 B^iB^jh_{ij}C_k=0&\Longrightarrow& B^iB^j(g_{ij}-l_il_j)C_k=0\\
 &\Longrightarrow& (B^2-B_0^2/F^2)C_k=0\\
 &\Longrightarrow& (B^2F^2-B_0^2)C_k=0\\
  &\Longrightarrow&  C_k=0, \text{ in view of } \eqref{lemmab}.
\end{eqnarray*}

\noindent\textbf{(d)} As $(M,F)$ is Berwaldian, $C_{ijk \mid h}=0$ and as $\det{({B^i}_{|j})}\neq 0$, then ${B^i}_{|j}\,Z^i_{k}= \delta^{i}_k$, where $(Z^i_{k})$ denotes the inverse of $({B^i}_{|j})$. Now, by the SC-condition, $B^i\,C_{ijk}=0$. Then we have the following implications:
\begin{eqnarray*}
 B^i\,C_{ijk}=0&\Longrightarrow& (B^i\,C_{ijk})_{\mid h}=0\\
 &\Longrightarrow& {B^i}_{\mid h}\,C_{ijk}=0\\
 &\Longrightarrow&  Z^h_r\, {B^i}_{\mid h}\,C_{ijk}=0\\
  &\Longrightarrow& \,C_{rjk}=0.
\end{eqnarray*}
This completes the proof.
\end{proof}

\begin{corollary}
If $B^i$ is a concurrent vector field, then (d) is true without any condition, since in this case ${B^i}_{|j}=-\delta^i_j$. Moreover, the same result is also true for the Landsbergian case. This retrieves some results of Matsumoto \em{\cite{meguchi}}.
\end{corollary}

\begin{remark}
\em{
It is to be noted that part \textbf{(a)} of the above theorem generalizes the main result of \cite{DWF}. The last is retrieved from  \textbf{(a)} by letting $B^i$ be a gradient of a non-constant function on M.
}
\end{remark}

\begin{theorem}\label{C-change7}
 A  semi-C-reducible Finsler manifold $(M,F)$ satisfying the SC-condition $(\ref{SC})$
  is either Riemannian or C2-like.
\end{theorem}

\begin{proof}
It is to be noted first that the condition $B^i \,C_{ijk}=0$  leads to $B^i\,C_i=0$. Now, contracting  (\ref{semi-c-red.}) by $B^iB^j$ and using the fact that $B^2F^2-B_0^2\neq 0$ \eqref{lemmab}, we get $ rB^iB^jh_{ij}C_k=0,$ hence,\vspace{-7pt}
\begin{eqnarray*}
 rB^iB^jh_{ij}C_k=0&\Longrightarrow& rB^iB^j(g_{ij}-l_il_j)C_k=0\\
 &\Longrightarrow& r(B^2-B_0^2/F^2)C_k=0\\
 &\Longrightarrow& r(B^2F^2-B_0^2)C_k=0\\
  &\Longrightarrow& r C_k=0,
\end{eqnarray*}
then either  $r=0$, which implies that  the space
is $C_2$-like,  or $C_i=0$, which implies that the  space is Riemannian.
\end{proof}

\begin{remark}
\em{
It should be noted that if $g$ is not positive definite, then the condition $C_i=0$ does not necessarily imply that $(M,\,F)$ is Riemannain \cite{deicke}. This can be shown by the following example (where the calculations have been performed using Maple program \cite{NSMaple}).
}
\end{remark}

Take $M=\mathbb{R}^3$, and  $F=f(x)(y_1\,y_2\,y_3)^{\frac{1}{3}}$. The Finsler function $F$ is not defined on the whole $T\mathbb{R}^3$, it is defined on the domain $D= T\mathbb{R}^{3} - \{(x_{i},y_{i}) \in T\mathbb{R}^3\, |\,y_{i}\neq 0 \}$.\\
The components of the metric are:
\begin{eqnarray*}
  g_{11} &=& -\frac{1}{9}\,\frac{f(x)\,(y_2\,y_3)^{2}}{\left(y_1\, y_2\, y_3 \right)^\frac{4}{3}},
             ~~~~g_{12} = \frac{2}{9}\,\frac{f(x)\,y_1\,y_2\, y_3^{2}}{\left(y_1\, y_2\, y_3 \right)^\frac{4}{3}}, \\
  g_{13} &=& \frac{2}{9}\,\frac{f(x)y_1\, y_2^{2}\, y_3}{\left(y_1\, y_2\, y_3 \right)^\frac{4}{3}},~~~~~~~g_{22} =- \frac{1}{9}\,\frac{f(x)\,(y_1\,y_3)^{2}}{\left(y_1\, y_2\, y_3 \right)^\frac{4}{3}},\\
  g_{23} &=&\frac{2}{9}\,\frac{f(x)y_1^{2}\, y_2\, y_3}{\left(y_1\, y_2\, y_3 \right)^\frac{4}{3}},~~~~~~~g_{33} =- \frac{1}{9}\,\frac{f(x)\,(y_1\,y_2)^{2}}{\left(y_1\, y_2\, y_3 \right)^\frac{4}{3}}.
\end{eqnarray*}
Hence, the components of Cartan tensor are:
\begin{eqnarray*}
   C_{111} &=& \frac{2}{27}\,\frac{f(x)\,(y_2\,y_3)^{3}}{\left(y_1\, y_2\, y_3 \right)^\frac{7}{3}}, \ ~~~~~~~~~C_{112} = -\frac{1}{27}\,\frac{f(x)\,y_1\,y_2^{2}\,y_3^{3}}{\left(y_1\, y_2\, y_3 \right)^\frac{7}{3}}, \\
  C_{113} &=& -\frac{1}{27}\,\frac{f(x)\,y_1\,y_2^{3}\,y_3^{2}}{\left(y_1\, y_2\, y_3 \right)^\frac{7}{3}}, \ ~~~~~~~C_{122} = -\frac{1}{27}\,\frac{f(x)\, y_1^{2} y_2\, y_3^{3}}{\left(y_1\, y_2\, y_3 \right)^\frac{7}{3}}, \\
  C_{123} &=& \frac{2}{27}\,\frac{f(x)\,(y_2\,y_3)^{2}\,y_1^{3}}{\left(y_1\, y_2\, y_3 \right)^\frac{7}{3}},~~~~~~~C_{133} = -\frac{1}{27}\,\frac{f(x)\,y_1^{2}\,y_2^{3}\,y_3}{\left(y_1\, y_2\, y_3 \right)^\frac{7}{3}}, \\
  C_{222} &=& \frac{2}{27}\,\frac{f(x)\,(y_1\,y_3)^{3}}{\left(y_1\, y_2\, y_3 \right)^\frac{7}{3}},~~~~~~~~~~~\!C_{223} = -\frac{1}{27}\,\frac{f(x)\,y_1^{3}\,y_2\,y_3^{2}}{\left(y_1\, y_2\, y_3 \right)^\frac{7}{3}}, \\
  C_{233} &=&  -\frac{1}{27}\,\frac{f(x)\,y_1^{3}\,y_2^{2}\,y_3}{\left(y_1\, y_2\, y_3 \right)^\frac{7}{3}},~~~~~~~~~C_{333} = \frac{2}{27}\,\frac{f(x)\,(y_1\,y_2)^{3}}{\left(y_1\, y_2\, y_3 \right)^\frac{7}{3}}.
\end{eqnarray*}
We note that, $C_i=0$, for all $i$, although the space is not Riemannian.

\vspace{5pt}
Moreover, in the above example, although $B^iC_i=0$, we do not have $B^i\,C_{ijk}=0$. For example,
$$B^i\,C_{i11}=\frac{1}{27}\,\frac{f(x)\,\left(2B^{1}\, y_2\,y_3 -B^{2}\, y_1\,y_3 -B^{3}\, y_1\,y_2 \right)}{y_1^{2}\,\left(y_1\, y_2\, y_3 \right)^\frac{1}{3}}\neq0.$$
Therefore,  the above space does not admit a semi-concurrent vector field.

\begin{remark}
\em{
As a by-product, the above example shows the necessity of the condition $F(-y)=F(y)$ in part $(2)$ of Theorem $\ref{matsumoto}$.
}
\end{remark}
The T-tensor is defined by \cite{r2.21}
$$T_{hijk}=F\,C_{hij}{\mid}_k
+C_{hij}\,l_k+C_{hik}\,l_j +C_{hjk}\,l_i+C_{ijk}\,l_h.$$
It is well-known  that if $(M,F)$ is Riemannian, then the T-tensor
vanishes. But the converse is not true in general. The next result
shows that the converse is true in the case where  $(M,F)$
satisfies the SC-condition $(\ref{SC})$.

\begin{theorem}\label{theoremC-change13}
A Finsler manifold satisfying the SC-condition is  Riemannian
if and only if the T-tensor  $T_{hijk}$ vanishes.
\end{theorem}

\begin{proof}
We first show that the vertical covariant derivative of $B^i$ vanishes identically. Indeed,
 $$B^i|_k=\dot{\partial}_kB^i+B^m\,C^i_{mk}=\dot{\partial}_kB^i=0,$$
 since $B^i$ are functions of $x$ only.

Let the T-tensor vanish, then
$$F\,C_{hij}|_k +C_{hij}\,l_k+C_{hik}\,l_j +C_{hjk}\,l_i+C_{ijk}\,l_h=0.$$
Contracting by $B^i$, and taking into account that $B^i|_k=0$, we find that
$\frac{B_0}{F}\,C_{hjk}=0$. Since $B_{0} \neq 0$ by \eqref{lemmab}, it follows that $C_{hjk}=0.$
\end{proof}

\vspace{5pt}
Let us write
\begin{equation*}\label{C-change14}
T_{ij}:=T_{ijhk}\,g^{hk}=F\,C_i|_j+l_i\,C_j+l_j\,C_i.
\end{equation*}

We have the following immediate result.
\begin{corollary}\label{C-change13}
 A Finsler manifold satisfying  SC-condition is  Riemannian
   if and only if the  tensor $T_{ij}$  vanishes.
\end{corollary}
\begin{remark}
\em{
Theorem $\ref{theoremC-change13}$ and corollary $\ref{C-change13}$ generalize the corresponding results of  Masumoto-Eguchi \cite{meguchi} for concurrent vector fields.
}
\end{remark}


\section{Special case: Concurrent vector fields}
As far as the authors know the first two papers which introduced the concept of a concurrent vector field on Finsler manifolds are Tachibana \cite{Tachibana} and Masumoto-Eguchi \cite{meguchi}.
Tachibana claimed that a necessary and sufficient condition for a Finsler manifold to admit a concurrent vector field is that its line element is expressible in the form
\begin{equation}\label{Tachibana-form}
  ds^2=(dx^n)^2+(x^n)^2H(x^\alpha,dx^\alpha), \,\,\, \alpha=1,..., n-1,
\end{equation}
where $H(x^\alpha,dx^\alpha)$ is the square of the line element of an arbitrary $(n-1)$-dimensional Finsler manifold.
\begin{theorem}\em{[Tachibana]}\\
A Finsler manifold $(M,F)$ admits a concurrent vector field if and only if $F$ satisfies $(\ref{Tachibana-form})$.
\end{theorem}

Matsumoto and Eguchi \cite{meguchi} remarked however that the proof of Tachibana's theorem is not clear.  In his book \cite{Matsumoto-book},  Matsumoto argued that, for metrics of the form \eqref{Tachibana-form}, the vector field $(0,...,0,-X^n)$ is certainly concurrent, but he could not see that the necessity of Tachibana's theorem should hold. In the following we prove that the form \eqref{Tachibana-form} implies that the metric is actually Riemannain.

\begin{theorem}\label{our-Tachibana-theorem}
A (regular) Finsler metric of the form $(\ref{Tachibana-form})$ is  Riemannian.
\end{theorem}
\begin{proof}
Equation $(\ref{Tachibana-form})$ can be written, in terms of the energy function $E$, in the form
$$E=(y^n)^2+(x^n)^2H(x^\alpha,y^\alpha), \,\,\, \alpha=1,..., n-1.$$
Since $E$ is a Finslerian energy function, then it is smooth on the whole of $\T M$ and particularly on the direction $(0,...,0,y^n)$. Consequently, $H$ is also smooth on $\T M$ and particularly on the direction $(0,...,0,y^n)$.
But $H$  does not depend on $y^n$ and hence the section  $(0,...,0,y^n)(\equiv \{(0_{x},...,0_{x},y^n);\, x \in M\})$  can be identified with the zero section of the $(n-1)$-dimensional  space. Now, $H$ is smooth, and particularly $C^2$ on $(0,...,0,y^n)\equiv (0,...,0) $ and homogenous of degree $2$, then $H$ is a polynomial of degree $2$. Hence, $H$ is quadratic in $y$, which means that  $E$ is Riemannian.
\end{proof}

As a direct consequence of Theorem $\ref{our-Tachibana-theorem}$, we have
\begin{theorem}\label{ourthmend}
Assuming that Tachibana's theorem is true, a Finsler metric admitting a concurrent vector field is Riemannian. Consequently, there is no regular Finsler non-Riemannian metric admitting a concurrent vector field.
\end{theorem}

A natural question arises:
\textit{Is a conic Finsler metric of the form \eqref{Tachibana-form}  admitting a concurrent vector field Riemannian?}
The following example gives a negative answer. It gives a non-Riemannian conic Finsler metric of the form \eqref{Tachibana-form} which admits a concurrent vector field.

\begin{example}
\em{
Let $M=\Real^3$ and $E$ be given by
$$E=y_3^2 + x_3^2H=y_3^2 +x_3^2\left(\sqrt{y_1^2+x_1^2y_2^2}+\epsilon y_2\right)^2.$$
where $H$ is a 2-dimensional Finsler metric of Randers type given by
$$H=\sqrt{y_1^2+x_1^2y_2^2}+\epsilon y_2.$$
One can easily show that $E$ is not smooth on the directions $(0,0,\pm 1)$. Hence, $E$ is defined on the conic set $D\subset TM$,
$$D=TM-\{(x_i,y_i) \in TM | \, y^2_{1}+y^2_{2}\neq 0 \} .$$
 The components $g_{ij}$ of the metric tensor  are given by
$$g_{11}=\frac{x_3^2\, (\epsilon\, x_1^4\, y_2^5 +\epsilon\, x_1^2 \,y_1^2 \,y_2^3 +\sqrt{y_1^2+x_1^2\,y_2^2}\,(x_1^2 \,y_2^2\,(y_1^2+x_1^2\,y_2^2) - x_1^2\, y_1^2\, y_2^2+2\, y_1^2\, (y_1^2+x_1^2\,y_2^2)- y_1^4))}{(y_1^2+x_1^2\,y_2^2)^{5/2}},$$
$$g_{22}=\frac{x_3^2\, (2\, \epsilon \,x_1^4\, y_2^3 +3\,\epsilon \,x_1^2\, y_1^2\, y_2 +x_1^2\,(y_1^2+x_1^2\,y_2^2)^{3/2} +\epsilon^2 \,(y_1^2+x_1^2\,y_2^2)^{3/2} )}{(y_1^2+x_1^2\,y_2^2)^{3/2}},$$
$$g_{12}=\frac{\epsilon \,x_3^2\, y_1^3 }{(y_1^2+x_1^2\,y_2^2)^{3/2}},\qquad g_{33}=1.$$
The components $C_{ijk}$ of the Cartan tensor are given by
$$C_{111} = -\frac{3}{2}\frac{\epsilon\, x_1^2\, x_3^2\, y_1 \, y_2^3  }{(y_1^2+x_1^2\,y_2^2)^{5/2}}, ~~~~~~~~~~~~~~
  C_{112}=\frac{3}{2}\frac{\epsilon\, x_1^2\, x_3^2 \,y_1^2 \, y_2^2  }{(y_1^2+x_1^2\,y_2^2)^{5/2}}$$
$$~~C_{122} = -\frac{3}{2} \frac{\epsilon \,x_1^2\, x_3^2 \,y_1^3  \, y_2}{(y_1^2+x_1^2\,y_2^2)^{5/2}}, ~~~~~~~~~~~~~~
  C_{222}=\frac{3}{2} \frac{\epsilon\, x_1^2\, x_3^2 \,y_1^4 }{(y_1^2+x_1^2\,y_2^2)^{5/2}}.$$
This metric admits a concurrent vector field given by $B^1(x)=B^2(x)=0,B^3(x)=x_3$.
Moreover, if we let $B^3=f(x)$, an arbitrary function of $x$, then $B=(0,0,f(x))$ is a semi-concurrent vector field. Clearly, the given metric is not Riemannian.
}
\end{example}


\section{Examples in dimension 4}
As has been shown, the problem mentioned in the introduction is completely solved for the $2$-dimensional case and also for several specific cases, where the Finsler manifold under consideration is subject to certain conditions. It turns out that in the general case ($\dim M\ge3$ and no additional restrictions), a
Finsler metric of the form \eqref{Tachibana-form} admitting a concurrent vector field is necessarily Riemannian, whereas a \emph{conic} Finsler metric of the the same form \eqref{Tachibana-form} admitting a concurrent vector field is not necessarily Riemannian. In what follows we present some examples of Riemannian and conic Finslerian metrics admitting semi-concurrent vector fields. In the examples considered all calculations are preformed using Maple program \cite{NSMaple}.

\vspace{7pt}
 Let us consider the manifold $M= \mathbb{R}^4$. A general form of a Finsler metric admitting a semi-concurrent vector field is given by:
\begin{eqnarray}\label{general-form}
E&=&y_1(F_1(x)y_2+y_4)F_2\left(\frac{(A_1x_1+A_2x_2+A_3x_3+A_4x_4+A_7)y_1+A_5y_2+A_6y_3}{y_1}\right)\nonumber\\
&&+F_3\left(x, -\frac{A_5y_2+A_6y_3}{y_1}\right)y_1^2-F_4(x)(A_5y_2^2+A_6y_2y_3)+F_5(x)y_1^2\\
&&+F_6(x)y_1y_2+F_7(x)y_2^2+F_8(x)y_4^2,\nonumber
\end{eqnarray}
where $A_1,..., A_7$ are arbitrary constants and $F_1, ..., F_8$ are arbitrary smooth functions on $TM$ or a subset of $TM$ such that $E$ is an energy function.
\par
By appropriate choices of $A_1,..., A_7;F_1, ..., F_8$, the energy function $E$ may be Riemannian or conic (pseudo) Finslerian which admits semi-concurrent vector fields, as shown below.

To find the components of the required semi-concurrent vector field $B=(B_1,B_2,B_3,B_4)$, we first find the metric components $g_{ij}$ corresponding to the above energy function. From this we calculate the Cartan tensor components $C_{ijk}$. The required components of the semi-concurrent vector field $B$ are then obtained by solving the system of equations $B^hC_{hij}=0$. These turn out to be
\vspace{-3pt}
\begin{eqnarray*}
B^1&=&0\\
B^2&=&f(x);\,\,f(x)\neq0 \\
B^3&=&- \frac{A_5}{A_6}f(x)\\
B^4&=&-f(x)F_1(x).
\end{eqnarray*}
The next examples, corresponding to different choices of the $A_{i}$'s and $F_{i}$'s represent some special classes of \eqref{general-form}.

\begin{example}
\em{
Set
  $$A_1=A_2=A_3=A_4=A_7=0;\,\, F_{1}(x)=F_4(x)=0,\,F_2(u)=u,\,F_3(x,u)=u^2$$
  in \eqref{general-form} to obtain
 $$E=y_4(A_5y_2+A_6y_3)+(A_5y_2+A_6y_3)^2+F_5(x)y_1^2+F_6(x)y_1y_2+F_7(x)y_2^2+F_8(x)y_4^2.$$
 This choice yields the metric components

 $$g=\left(
   \begin{array}{cccc}
     F_5(x) & \frac{1}{2}F_6(x) & 0 & 0 \\
     \frac{1}{2}F_6(x) & A_5^2+F_7(x) & A_5A_6 & \frac{1}{2}A_5 \\
     0 & A_5A_6 & A_6^2 &  \frac{1}{2}A_6 \\
     0 & \frac{1}{2}A_5 & \frac{1}{2}A_6 & F_8(x) \\
   \end{array}
 \right).$$
 Clearly, the above matrix has rank 4, and so $g$ is non-degenerate. Consequently, $g$ is a pseudo-Riemannian metric. This metric is positive definite if the leading principal minors of the above matrix are all positive. By some  computations, the leading principal minors of $g$ are
 $$F_5, \quad A_5^2F_5+F_5F_7-\frac{F_6^2}{4}, \quad \frac{A_6^2}{4}(4F_5F_7-F_6^2), \quad \frac{A_6^2}{16}(4F_5F_7-F_6^2)(4F_8-1).$$
 Hence, the metric $g$ is is Riemannian if $F_5>0, \,\, F_5F_7>\frac{F_6^2}{4}, \,\, F_8>\frac{1}{4}$.\\
}
\end{example}

\vspace{-0.8cm}
\begin{example}
\em{
Set
 $$A_1=A_2=A_3=A_4=A_5=A_7=0;\,\, F_1(x)=F_4(x)=0,\,\, F_2(u)=u,\,\, F_3(x,u)=u^4$$
 in \eqref{general-form} so that
 $$E=A_6y_3y_4+\frac{A_6^4\,y_3^4}{y_1^2}+F_5(x)y_1^2+F_6(x) y_1y_2+F_7(x)y_2^2+F_8(x)y_4^2.$$
 This energy function represents  a conic Finslerian metric whose conic domain has the form
 $$D=\{(x,y) \in TM \mid y_1 \neq 0\}.$$
 The metric $g$ is given by
 $$g=\left(
   \begin{array}{cccc}
     \frac{3A_6^4y_3^4+F_5(x)y_1^4}{y_1^4} & \frac{1}{2}F_6(x) &- \frac{4A_6^4y_3^3}{y_1^3} & 0 \\
     \frac{1}{2}F_6(x) & F_7(x) & 0 & 0 \\
     - \frac{4A_6^4y_3^3}{y_1^3}& 0 & \frac{6A_6^4y_3^2}{y_1^2} &  \frac{1}{2}A_6 \\
     0 & 0 & \frac{1}{2}A_6 & F_8(x) \\
   \end{array}
 \right).$$
 As the matrix $g$ has rank 4, the metric tensor $g$ is thus pseudo-Finslerian. It can be shown that the leading principal minors of $g$ are:
 $$\frac{3A_6^4y_3^4+5y_1^4}{y_1^4}, \quad \frac{12A_6^4 F_7y_3^4-F_6^2y_1^4+20F_7y_1^4}{4y_1^4}, \quad
 \frac{A_6^4y_3^2(4A_6^4F_7y_3^4+3F_6^2y_1^4+24F_6y_1^3y_3+60F_7y_1^4)}{2y_1^6}, $$
 $$\frac{A_6^2(32A_6^6F_7F_8y_3^6-12A_6^4F_7y_1^2y_3^4-24A_6^2F_6^2F_8y_1^4y_3^2+480A_6^2F_7F_8y_1^4y_3^2+F_6^2y_1^6-20F_7y_1^6)}{16y_1^6}.$$
Therefore, $g$ is conic Finslerian if the above leading principal minors are all positive.
The non-vanishing components of Cartan tensor are
$$C_{111} = -\frac{6\,A_6^4\,y_3^4}{y_1^5}, \,\,\,  C_{113} = \frac{6\,A_6^4\,y_3^3}{y_1^4}, \,\,\,
C_{133} =- \frac{6\,A_6^4\,y_3^2}{y_1^3}, \,\,\,  C_{333} = \frac{6\,A_6^4\,y_3}{y_1^2}$$
In this case, the semi-concurrent vector fields $B$ is given by
$$B^1=B^3=B^4=0,\,\,B^2=f(x);\,\,f(x)\neq0,$$
with $B^hC_{hij}=0$.
}
\end{example}

\vspace{5pt}
\begin{example}
\em{
The following choice of some arbitrary constants and functions in \eqref{general-form}, namely
$$A_1=A_2=A_3=A_4=A_7=0;\,\, F_2(u)=u^2,\,\, F_3(x,u)=u,\,\,  F_4=0,$$
represents a more nontrivial example of a conic pseudo-Finsler metric defined by
$$E= \frac{({F_1}(x)\,y_2+y_4)(A_5\,y_2+A_6\,y_3)^2}{y_1}-(A_5\,y_2+A_6\,y_3)y_1+{F_5}(x)\,y_1^2+{F_6}(x)\,y_1\,y_2+{F_7}(x)\,y_2^2+{F_8}(x)\,y_4^2.$$
The non-vanishing components of the metric tensor are given by
\begin{eqnarray*}
g_{11 } &=&\frac{{ A_5}^{2}{y_2}^{2}({F_1}(x)\,y_2 +y_4)+2 A_5\, A_6\,{F_1}(x) {y_2}{y_3}(y_2+y_4) +{ A_6}^{2}\,{ y_3}^{2}(\,{F_1}(x)\, y_2 +y_4)+5{y_1}^{3}}{{y_1}^{3}} \\
g_{12 } &=& -\frac{1}{2}\,\frac{{ A_5}^{2}\,y_2(3{F_1}(x)\,y_2 +2\,y_4)+4 A_5\, A_6\,{F_1}(x) {y_2}\,{y_3}+{ A_6}^{2}\,{F_1}(x)\, {y_3}^{2}  +2{A_5}\,A_6\, y_3\,y_4  -{ F_6}(x)\,{y_1}^{2}}{{y_1}^{2}}\\
g_{22 } &=& \frac{3{A_5}^{2}\,{F_1}(x)\, y_2 +2\,A_5\,A_6\,{F_1}(x)\,y_3++{A_5}^{2}\,y_4 +{F_7}(x)\, y_1 }{ y_1} \\
g_{13 } &=& -\frac{1}{2}\frac{A_6\left(2A_5{ F_1}(x)\, { y_2}^2 +2{A_6}\,F_1(x)\,{ y_2}y_3 +2{ A_5} \,{ y_2}\,y_4 +2{ A_6} \,{ y_3}\,y_4+{y_1}^{2}\right) }{ {y_1}^{2}},\\
g_{23} &=& \frac{A_6\,\left(2 A_5\, { F_1}(x) \,{y_2} +A_6\,{F_1}(x)\, y_3 +A_5\,y_4 \right)}{y_1},~~~~~~g_{14 } = -\frac{1}{2}\,\frac{\left(A_5\, y_2 + A_6\, y_3 \right)^{2}}{{y_1}^{2}} \\
g_{24} &=& \frac{A_5\left(A_5\, y_2 +A_6\, y_3 \right)}{y_1},~~~~~~~~~~~~~~~~~~~~~~~~~~~~~~~~~~
            g_{33} = \frac{{A_6}^{2}\,\left({F_1}(x)\,y_2+y_4 \right)}{y_1} \\
g_{34} &=& \frac{A_6\left(A_5\,y_2 + A_6\, y_3 \right)}{y_1},~~~~~~~~~~~~~~~~~~~~~~~~~~~~~~~~~~
            g_{44} = {F_8}(x)
\end{eqnarray*}
Therefore, the non-vanishing components of Cartan tensor are given by
\begin{eqnarray*}
C_{111} &=& -\frac{3}{2}\,\frac{A_5\,{ y_2}^{2}({F_1}(x)[A_5 y_2+2 A_6  y_3]+A_5y_4)+{ A_6}^{2}\,{y_3}^{2}({F_1}(x)\,y_2 +y_4 )+2\, A_5\, A_6\,y_2\,y_3\, y_4 }{{y_1}^{4}} \\
C_{112} &=& \frac{1}{2}\,\frac{3 A_5^2\,{ F_1}(x)\, { y_2}^{2}+4\,A_5\,A_6\, {F_1}(x)\, y_2\,y_3 +{A_6}^{2}\,{ F_1}(x) \,{ y_3}^{2}+2{ A_5}^{2}\,y_2\, y_4 +2\, A_5\,A_6\,y_3\,y_4}{{ y_1}^{3}} \\
C_{113} &=& \frac{{ A_6} \,\left(A_5\, { F_1}(x) { y_2}^{2}+A_6 \,{ F_1}(x)\, y_2\, y_3 +A_5\,y_2\, y_4 +A_6\,y_3\, y_4 \right)}{{ y_1}^{3}} \\
C_{122} &=& -\frac{1}{2}\,\frac{A_5\,\left(3\, A_5\, F_1(x)\, y_2 +2\, A_6 F_1(x)\, y_3 +A_5\, y_4 \right)}{y1^2} \\
C_{123} &=& -\frac{1}{2}\,\frac{A_6\,\left(2\, A_5\, F_1(x)\, y_2 +A_6\, F_1(x)\, y_3 + A_5\, y_4 \right)}{y_1^2}
\end{eqnarray*}

\begin{eqnarray*}
C_{124} &=& -\frac{1}{2}\,\frac{A_5\,\left( A_5\, y_2 +A_6\, y_3\right)}{y_1^2} ,~~~~~~~~~~~~~~~~~~~~~~~\,
             C_{133} = -\frac{1}{2}\,\frac{A_6^2(F_1(x)\, y_2 +y_4)}{y_1^2} \\
C_{134} &=&-\frac{1}{2}\,\frac{A_6\left( A_5\, y_2+ A_6\, y_3 \right)}{y_1^2}  ,~~~~~~~~~~~~~~~~~~~~~~~~
             C_{222} = \frac{3}{2}\,\frac{A_5^2\, F_1(x)}{y_1} \\
C_{223} &=& \frac{A_6\, A_5\, F_1(x)}{y_1} ,~~~~~~~~~~~~~~~~~~~~~~~~~~~~~~~~~~~~~
             C_{224} = \frac{1}{2}\,\frac{A_5^2}{y_1} \\
C_{233} &=& \frac{1}{2}\,\frac{A_6^2\, F_1(x)}{y1},~~~~~~~~~~~~~~~~~~~~~~~~~~~~~~~~~~~~~~~\!
             C_{234} = \frac{1}{2}\,\frac{A_6\, A_5}{y_1} \\
C_{114} &=&\frac{1}{2}\,\frac{A_5^{2}{ y_2}^{2}+2\,A_5\,A_6\,y_2\,y_3 +{A_6}^{2}\,{ y_3}^{2}}{{ y_1}^{3}} ,~~~~~~~~~~~\!\!
             C_{334} = \frac{1}{2}\,\frac{A_6^2}{y1}.
\end{eqnarray*}
The semi-concurrent vector field $B$ is given by
$$B(x)=f(x)\frac{\partial}{\partial x^2}-\frac{A_5}{A_6}f(x)\frac{\partial}{\partial x^3}-f(x)F_1(x)\frac{\partial}{\partial x^4}, \,\, f(x)\neq0.$$
Note that in this example most of the components of Cartan tensor and three of the components of the vector field $B$ are alive.
}
\end{example}

\vspace{5pt}
All the above examples are shown to be either Riemannian or conic (pseudo) Finslerian. The only two choices in \eqref{general-form} that produce a regular metric are the following: $A_5=A_6=0$ or $F_2(u)=u,\, F_3(x,u)=u^2$. But these choices  yield a quadratic energy, which means that the metric is Riemannian. We conclude that, in dimension 4, no choice of $A_i$ and $F_i$ in \eqref{general-form} can yield a regular Finsler metric.

All the examples presented in this paper, among other evidences, motivate us to announce the following conjecture.

\section*{Conjuncture}

There is no regular Finsler non-Riemannian metric that admits a semi-concurrent vector field. In other words, a Finsler metric admitting a semi-concurrent vector field is necessarily either Riemannian or conic Finslerian.


\providecommand{\bysame}{\leavevmode\hbox
to3em{\hrulefill}\thinspace}
\providecommand{\MR}{\relax\ifhmode\unskip\space\fi MR }
\providecommand{\MRhref}[2]{%
  \href{http://www.ams.org/mathscinet-getitem?mr=#1}{#2}
} \providecommand{\href}[2]{#2}


\begin{thebibliography}{1}


\bibitem{Berwald41}
 L. Berwald, \emph{On Finsler and Cartan geometries. III. Two-dimensional Finsler spaces with rectilinear extremals}, Ann. Math., \textbf{42} (1941), 84--112.

\bibitem{deicke}
A. Deicke, \emph{$\ddot{U}$ber die Finsler-r$\ddot{a}$ume mit
$A_i=0$,} Arch. Math., \textbf{4} (1953), 45-51.

\bibitem{DWF}
M. Faghfouri and R. Hosseinoghli, \emph{A note on 2-dimentional Finsler manifold}, J. Adv. Res. Pure Math., \textbf{7} (3) (2015), 1-6.
arXiv: 1509.09172v1 [math.DG].

\bibitem{hashiguchi}
M.~Hashiguchi, \emph{On conformal transformations of  Finsler
metrics},
  J. Math. Kyoto Univ., \textbf{16} (1976), 25--50.

\bibitem{conic}
M. Javaloyes and M. Sanchez, \emph{On the definition and examples of Finsler metrics,} Ann. Sci. Norm. Super. Pisa, Cl. Sci. (5),  \textbf{13} (2014), 813--858.





\bibitem{Matsumoto-book}
M.~Matsumoto, \emph{The theory of Finsler connections}, Publications of the study Group of Geometry, \textbf{5}, (1970), Okayama University.

\bibitem{r2.9}
M.~Matsumoto, \emph{On C-reducible Finsler spaces }, Tensor, N.
S., \textbf{24} (1972), 29--37.



\bibitem{r2.22}
M.~Matsumoto, \emph{Foundations of Finsler geometry and special \textsc{F}insler spaces}, Kaiseisha Press, Otsu, Japan, 1986.

\bibitem{meguchi}
M.~Matsumoto and K.~Eguchi, \emph{Finsler spaces admitting a
concurrent vector  field}, Tensor, N. S., \textbf{28} (1974),
239--249.

\bibitem{C_2-like}
M. Matsumoto and S. Numata, \emph{On semi-C-reducible Finsler spaces with constant coefficients and $C_2$-like Finsler spaces}, Tensor, N. S., \textbf{34} (1980), 218--222.

\bibitem{r2.21}
M.~Matsumoto and C.~Shibata, \emph{On
semi-$\textsc{C}$-reducibility,
  $\textsc{T}$-tensor $=0$ and $\textsc{S}_{4}$-likeness of Finsler
  spaces}, J. Math. Kyoto Univ., \textbf{19} (1979), 301--314.

\bibitem{Petrovic} M. Petrovic, R. Rosca and L. Verstraelen, \emph{Exterior concurrent vector fields on a Riemannian manifold I. Some general results}, Soochow J. Math., \textbf{15} (1989),  179--187.

\bibitem{DW_Tayebi}
E. Peyghan and  A. Tayebi ,  \emph{On doubly warped product Finsler manifolds}, Nonlinear
Anal. Real World Appl., \textbf{13} (2012), 1703--1720.

\bibitem{Sasaki}S. Sasaki, \emph{On the structure of Riemannian spaces, whose group of holonomy fix a point or a direction}, Nippon Sugaku Butsuri Gakkaishi, \textbf{16} (1942), 193--200.

\bibitem{Tachibana}
 S. Tachibana, \emph{On Finsler spacess which admit a concurrent vector field}, Tensor, N.S., \textbf{1} (1950), 1--5.

\bibitem{Yano} K. Yano, \emph{Sur le parall\'{e}lism et la concourance dans l'espace de Riemann}, Proc. Imp. Acad. Tokyo, \textbf{19} (1943), 189--197.


\bibitem{beta2}
Nabil~L. Youssef, S.H. Abed and S.G.  Elgendi, \emph{ Generalized $\beta$-conformal change and special Finsler
spaces}, Int. J. Geom. Methods Mod. Phys., \textbf{9 (3)} (2012), 1250016, 25 pages. DOI:10.1142/S0219887812500168. arXiv:1004.5478 [math.DG].

\bibitem{NSMaple}
Nabil~L. Youssef and S.~G. Elgendi, \emph{New Finsler package}, Comp. Phys. Commun., \textbf{185} (2014), 986-997.
DOI: 10.1016/j.cpc.2013.10.024. arXiv:1306.0875 [math.DG].

\end{thebibliography}
\end{document}